\documentclass[a4paper]{amsart}
\usepackage{amsmath}
\usepackage{amssymb}
\usepackage{amsfonts}
\usepackage{amsthm}
\usepackage{amssymb}
\usepackage{amsbsy}
\usepackage{color}
\usepackage{graphicx}
\usepackage{subfigure}
\usepackage[utf8]{inputenc}
\usepackage{csquotes}
\usepackage[greek,english]{babel}

\usepackage{enumitem}
\RequirePackage[normalem]{ulem}
\usepackage[makeroom]{cancel}

\usepackage[bookmarks]{hyperref}

\newtheorem{theorem}{Theorem}[section]
\newtheorem{corollary}[theorem]{Corollary}
\newtheorem{lemma}[theorem]{Lemma}
\newtheorem{proposition}[theorem]{Proposition}
\newtheorem{assumption}[theorem]{Assumption}

\newtheorem{problem}[theorem]{Problem}
\theoremstyle{definition}
\newtheorem{definition}[theorem]{Definition}
\theoremstyle{remark}
\newtheorem{remark}[theorem]{\textbf{Remark}}
\newtheorem{example}[theorem]{\textbf{Example}}
\numberwithin{equation}{section}

\newcommand{\E}{\mathbb{E}}

\renewcommand{\P}{\mathbb{P}}
\newcommand{\Q}{\mathbb{Q}}

\newcommand{\di}{\mathrm{d}}

\newcommand{\R}{\mathbb{R}}
\newcommand{\Rp}{\R_+}

\newcommand{\Lc}{\mathcal{L}}
\newcommand{\as}{a.s.}

\newcommand{\dx}{\mathrm{d}x}

\newcommand{\ds}{\mathrm{d}s}

\newcommand{\dt}{\mathrm{d}t}

\newcommand{\F}{\mathcal{F}}
\newcommand{\Fb}{\mathbb{F}}
\newcommand{\Fc}{\mathcal{F}}
\newcommand{\Gc}{\mathcal{G}}
\newcommand{\Gb}{\mathbb{G}}
\newcommand{\Tc}{\mathcal{T}}
\newcommand{\Wb}{\mathbb{W}}

\newcommand{\ind}{\boldsymbol{1}}

\DeclareMathOperator{\supp}{supp}

\newcommand{\RST}{\mathsf{RST}}
\newcommand{\MVM}{\mathsf{MVM}}

\newcommand{\Ep}[1]{\E\left[#1\right]}

\renewcommand{\vec}[1]{\mathbf{#1}}

\newcommand{\Nb}{\mathbb{N}}

\newcommand{\Rc}{\mathcal{R}}
\newcommand{\Pc}{\mathcal{P}}
\newcommand{\Ac}{\mathcal{A}}

\newcommand{\Bc}{\mathcal{B}}
\newcommand{\Wc}{\mathcal{W}}
\newcommand{\cadlag}{c\`adl\`ag}

\usepackage{xcolor}

\definecolor{orange}{rgb}{1,0.3,0.2}

\newcommand{\Db}{\mathbb{D}}
\newcommand{\Pk}{\mathfrak{P}}

\title[A DPP for distribution-constrained optimal stopping]{A Dynamic Programming Principle for distribution-constrained optimal stopping}

\author{Sigrid K\"{a}llblad}
\thanks{
Technische Universität Wien, Wiedner Hauptstrasse 8-10, 1040 Vienna, Austria.
\\ e-mail: \texttt{sigrid.kaellblad@tuwien.ac.at}}

\date{\today}

\begin{document}

\begin{abstract}
	We consider an optimal stopping problem where a constraint is placed on the distribution of the stopping time. Reformulating the problem in terms of so-called measure-valued martingales allows us to transform the marginal constraint into an initial condition and view the problem as a stochastic control problem; we establish the corresponding dynamic programming principle. 
\end{abstract}

\maketitle

\section{Introduction}

	We consider the following optimal stopping problem with a constraint placed on the distribution of the stopping time: given a probability measure $\mu$ on $(0,\infty)$ and a filtered probability space supporting a Brownian motion $(B_t)_{t\ge 0}$, we aim at finding
		\begin{align}\label{intro:problem}
			\sup_{\tau\in\Tc(\mu)}\Ep{c\big(B_\cdot,\tau\big)},
		\end{align}
	where $\Tc(\mu)$ is the set of stopping times with distribution $\mu$ and $c$ is a given measurable cost function satisfying $c(\omega,t)=c(\omega_{\cdot\wedge t},t)$.
	The problem is related to the so-called inverse first-passage-time problem which has a long history; it has also attracted recent attention: see e.g. \cite{bayraktar2016,manu,ekstroem2016}, to which we refer for further motivation, references and an exposition of its role within financial and actuarial mathematics. 
	
	When the underlying filtration is general enough to allow for an independent uniformly distributed random variable, problem \eqref{intro:problem} is equivalent to its weak formulation where the supremum is also taken over potential probability spaces. This observation underlies the approach in Beiglb\"ock et al. \cite{manu}; specifically, identifying stopping times with measures on a certain canonical product space, the existence of an optimiser along with a monotonicity principle characterising the support set of any optimiser is obtained. In turn, for certain classes of cost functions, the latter is used to deduce the existence of so-called barrier-type solutions. 
	Herein, assuming that the filtration is general enough in the above sense, and imposing certain regularity but notably no specific structure on the cost-function, we reformulate the problem in terms of so-called measure-valued martingales which enables addressing the problem as a stochastic control problem; we establish the corresponding dynamic programming principle.
	
	The notion of measure-valued martingales was used by Cox and K\"allblad \cite{cox2015} to study a robust pricing problem where optimisation takes places over a class of martingales (potential market models) satisfying a given terminal marginal constraint (guaranteeing fit to given market data). Each martingale was then identified with the process specifying the conditional distribution of its terminal value. The latter, naturally, belongs to a class of processes taking values in the space of probability measures on $\R$ and satisfy a certain martingale property; they are referred to as measure-valued martingales (MVMs). Crucially, the terminal marginal constraint was then transformed into an initial condition for the corresponding MVMs, which allowed for the problem to be viewed as a stochastic control problem where the conditional law of the underlying appeared as an additional state-process. The problem can thus be addressed via the dynamic programming approach, and in \cite{cox2015} the DPP was established and a HJB equation deduced for the value function. 
	
	The same approach may naturally be applied also to the optimal stopping problem \eqref{intro:problem} and the aim herein is to specify this and prove the associated DPP. Specifically, each stopping time $\tau$ in $\Tc(\mu)$ is identified with the MVM $(\xi_t)_{t\ge 0}$ defined as its conditional distribution given the current information:
	\begin{align*}
		\xi_t=\mathcal{L}(\tau|\Fc_t);
	\end{align*}		
	any such process will satisfy the initial condition $\xi_0=\mu$ along with a martingale property and a certain adaptedness condition corresponding to the stopping-time property of $\tau$.  
	When reformulating the optimal stopping problem as an optimisation problem over such MVMs, the distribution-constraint is then incorporated as an initial condition which allows the problem to be addressed as a stochastic control problem; our main result establishes that the dynamic programming principle holds for this problem. 
	
	We note that in Ankirchner et al. \cite{ankirchner2015} and Miller \cite{miller2015}, the optimal stopping problem was studied under a constraint on the expected value on the stopping time. The conditional expected value of the terminal constraint was then incorporated as an additional state-process and the problem addressed via the use of BSDEs (cf. also \cite{bouchard2012}); at an abstract level our approach thus bears a certain resemblance to theirs. 
	The distribution-constrained problem that we consider has also been studied from a similar perspective by Bayraktar and Miller \cite{bayraktar2016}. Therein, however, the authors consider the strong problem formulation and restrict to the case of atomic constraints whereas we let $\mu$ be a general probability measure on $\R_+$; we also allow for more general cost functions. 
	We note that our method of proof is distinct from theirs. Indeed, the approach in \cite{bayraktar2016} uses that any point in the unit simplex can be written as a linear combination of a finite number of points whose convex hull includes that point; a method which, to us, seems difficult to generalise to arbitrary constraints.
	Moreover, since they consider the strong formulation, their filtration is generated by the Brownian motion alone and additional randomness is obtained by conditioning on the past of the Brownian motion itself, an approach which seems difficult to generalise beyond non-path-dependent cost functions. 	
	Notably, however, as a consequence of our results herein combined with theirs, we obtain that for the class of cost functions considered in \cite{bayraktar2016}, the problem when restricting to the Brownian filtration coincides with the weak formulation. Hence, a-posteriori, we recover their main result as a special case of ours. 
	
	To prove the DPP for (the MVM version of) problem \eqref{intro:problem} we here consider its weak formulation on the associated canonical path space; the canonical framework has previously successfully been used for the study of stochastic control problems by e.g. El Karoui and Tan \cite{karoui2013_1,karoui2013_2}, see also \cite{neufeld2013,nutz2013,zitkovic2014}.
	To account for the fact that we are dealing with measure-valued martingales, (a component of) our canonical space consists of (continuous) functions from $\R_+$ into the space of probability measures on $\R_+$, where we equip the latter with the topology induced by the first Wasserstein distance rendering it a Polish space.
	We may then establish the DPP by proving analyticity, and stability under concatenation and disintegration, of a certain set of measures, and, in turn, apply a version of Jankov-von Neumann's measurable selection theorem. 
	Underlying this approach lies the fact that problem \eqref{intro:problem} is equivalent to its weak formulation; in particular, although there are alternatives to this convention, we herein let the MVMs correspond to the conditional distribution of the (disintegrated) randomised stopping times. 
	
	A-priori assuming continuity of the value-function, and restricting to a class of cost functions admitting a certain (Markovian) structure, we also provide an alternative proof of the DPP. Although this result is a special case of our main result, we choose to report this independent argument for we find it of interest. Specifically, following Bouchard and Touzi \cite{bouchard2011} (see also \cite{bouchard2012}), the idea is to exploit the continuity properties of the value function in order to explicitly construct an approximately optimal kernel and thus circumvent the need for measurable selection arguments. 
	To deal with the measure-valued argument we here borrow ideas from optimal transport theory -- the argument crucially relies on the structure of the first Wasserstein distance. Notably, the class of cost functions for which this argument applies includes the ones considered in \cite{bayraktar2016}. 
	
	The DPP aside, we also study the stability of the distribution-constrained optimal stopping problem in that we establish continuity properties of the value of the problem as a function of the marginal constraint. First, following the arguments used to prove existence of an optimiser in \cite{manu}, we obtain upper semi-continuity of this mapping under rather general assumptions on the cost function. Imposing additional regularity on the cost function, we also establish a certain 'right-continuity' of the value function; the result is of interest since it ensures that the value of the problem with a general constraint can be obtained as the limit of a sequence of approximating atomic problems which are easier to address by use of numerical methods. 
	
	The remainder of the article is organised as follows: in Section 2, we introduce our problem of study and establish the continuity properties. In Section 3.1 we reformulate the problem in terms of so-called adapted measure-valued martingales; we provide the DPP in Section 3.2 and its proof in Section 3.3. The alternative proof of the DPP is deferred to the Appendix.

\section{The distribution-constrained optimal stopping problem}

	\subsection{The problem of study: definition and first remarks}
	
	Given a fixed law $\mu\in\Pc$, where $\Pc$ denotes the set of probability measures on $(0,\infty)$ with finite first moment, and a filtered probability space $(\Omega,\Gc,(\Gc_t)_{t\ge 0},\P)$ supporting a Brownian motion $(W_t)_{t\ge 0}$ and an independent $\Gc_0$-measurable random variable $U\sim U[0,1]$, we consider the following distribution-constrained optimal stopping problem:
		\begin{align}\label{problem}
			v(\mu)=\sup_{\tau\in\Tc(\mu)}\Ep{c\big(W_\cdot,\tau\big)},
		\end{align}
	where $\Tc(\mu)$ is the set of stopping times with given distribution $\mu$, and $c$ is a given measurable cost function from $C_0(\R_+)\times\R_+$ to $\R$ satisfying $c(\omega,t)=c(\omega_{\cdot\wedge t},t)$. We assume throughout that $\Ep{c(W_\cdot,\tau)}$ is well defined and $<\infty$ for all $\tau\in\cup_{\mu\in\Pc}\Tc(\mu)$; and that for all $\mu\in\Pc$, there exists some $\tau\in\Tc(\mu)$ such that $\Ep{c(W_\cdot,\tau)}>-\infty$. In particular, $v$ defined in \eqref{problem} is thus a function from $\Pc$ to $(-\infty,\infty]$.

	\subsubsection{Preliminaries on the weak problem formulation}\label{sec:preliminaries}
	
	For our approach it is crucial that the distribution-constrained problem \eqref{problem} is in fact equivalent to the corresponding weak formulation. To formalise this, following \cite{beiglboeck2016} and \cite{manu}, we will make use of the notion of randomised stopping times; see however also \cite{guo2016,kallblad2015,karoui2013_1,karoui2013_2}. 
	Let $C_0(\R_+)$ denote the space of continuous functions from $\R_+$ to $\R$ starting in zero, equipped with the topology of uniform convergence on compact sets; we denote the Wiener measure on $C_0(\R_+)$ by $\Wb$, and let the filtration $\Fb=(\Fc_t)_{t\ge 0}$ be the usual augmentation of the canonical filtration. 	
	 Following \cite{manu}, see also \cite{beiglboeck2016} to which we refer for further details, we then define the set $\RST(\mu)$ of probability measures on the product space $C_0(\R_+)\times\R_+$, referred to as randomised stopping times with pre-specified law\footnote{In \cite{beiglboeck2016}, sub-probability measures $\gamma$ with $\mathrm{proj}_{C_0(\R_+)}\gamma\le \Wb$ are considered. However, since we here require $\mathrm{proj}_{\R_+}\gamma=\mu$, where $\mu\in\Pc$ has mass $1$, it follows that so has $\gamma$ and it suffices to restrict to so-called 'finite' RST for which $\mathrm{proj}_{C_0(\R_+)}\gamma= \Wb$ and $\tau$ is finite for almost all paths; see also p. 10 in \cite{manu}.}:
			\begin{align*}
				\mathsf{RST}(\mu):=\{\gamma\in\mathsf{Cpl}(\mathbb{W},\mu): \textrm{$A^\gamma_t(\omega):=\gamma_\omega([0,t])$ is $\Fc_t$-measurable, for $t\ge 0$}\},
			\end{align*}					
		where $(\gamma_\omega)_{\omega\in C_0(\R_+)}$ denotes the disintegration of $\gamma$ in the first co-ordinate, and $A^\gamma(\omega)$ is the cumulative distribution function associated with $\gamma_\omega$.  
		
	Denoting by $(B,T)$ the canonical process on $C_0(\R_+)\times\R_+$, according to Lemma 3.11 in \cite{beiglboeck2016}, the distribution-constrained problem may then be viewed as an optimisation problem over randomised stopping times on the canonical space; specifically, given that $(\Omega,\Gc,(\Gc_t),\P)$ supports a Brownian motion and an independent uniformly distributed $\Gc_0$-random variable, problem \eqref{problem} is equivalent to the problem of maximising 
		\begin{align*}
			\E^\gamma\left[c\big(B_\cdot,T\big)\right], \;\;\textrm{over $\gamma\in\mathsf{RST}(\mu)$}. 
		\end{align*} 
	In particular, the specific choice of the probability space $(\Omega,\Gc,(\Gc_t),\P)$ has thus no bearing on the problem. Hence, denoting by $\Ac(\mu)$ the set of all terms of the form $\alpha=(\Omega^\alpha,\Gc^\alpha,\Gb^\alpha,\P^\alpha,(W^\alpha_\cdot),\tau^\alpha)$ such that $(\Omega^\alpha,\Gc^\alpha,\Gb^\alpha,\P^\alpha)$ is a probability space in which $W^\alpha_\cdot$ is a BM and $\tau^\alpha$ a stopping time with $\tau\sim\mu$, the value remains the same if optimising
	\begin{align}\label{eq:weak_original}
		\E^\alpha\big[\big(W^\alpha_\cdot,\tau^\alpha\big)] \;\; \textrm{over all terms $\alpha\in\Ac(\mu)$};
	\end{align} 
	this further illustrates that we are indeed dealing with a weak problem formulation.
	
	Since every randomised stopping time admits a characterisation in terms of its disintegration in the first variable, $(\gamma_\omega)_{\omega\in C_0(\R_+)}$, working on the fixed probability space $(\Omega=C_0(\R_+),\Fc,\Fb,\mathbb{W})$, still denoting the canonical process by $B$, our problem may also equivalently be formulated as optimising 
		\begin{align}\label{eq:problem_formulation_useful}
			\E\left[\int_0^\infty c(B_\cdot,s)\di A^\gamma_s\right],
			\;\;\textrm{with $A^\gamma_t(\omega):=\gamma_\omega([0,t])$},
		\end{align}
	over disintegrated kernels $(\gamma_\omega)$ of $\gamma\in\RST(\mu)$. This last formulation emphasises the following intuitive viewpoint on randomised stopping times: while a standard stopping time assigns a single time to stop to each path, a stopping time depending on some external randomisation may be viewed as a distribution specifying the probability to stop at various points given the observation of a specific path. 
	
	Since the above formulations are all equivalent, we are free to switch between them and will consider the formulation most convenient in each individual situation.

	\subsection{The dependence on the constraint: stability of the value function} 
	
	
	In this section we study the dependence of the problem on the marginal constraint in that we establish continuity properties of the mapping $\mu\mapsto v(\mu)$. Recall that the existence of an optimiser to Problem \eqref{problem} was established in \cite{manu}; a slight variation of their argument yields upper semicontinuity of the value function: 	 
	
	\begin{proposition}\label{lem:usc}
		Suppose that the cost function $c$ is bounded from above and that $t\mapsto c(\omega_{\cdot\wedge t},t)$ is upper semicontinuous for $\mathbb{W}$-a.e. $\omega\in\Omega$. 
		Then, $\mu\mapsto v(\mu)$ is upper semicontinuous on $\Pc$ in the topology induced by $\Wc_1$, the first Wasserstein metric.
	\end{proposition}
	
	\begin{proof}[Proof of Proposition \ref{lem:usc}]
		Let $(\mu_n)$ be a sequence converging to some $\mu$; w.l.o.g. let $\Wc_1(\mu_n,\mu)$ non-increasing in $n$. In turn, let $(\gamma_n)$ a sequence such that $\gamma_n\in\mathsf{RST}(\mu_n)$ and
			\begin{align}\label{eq:1}
				\lim_{n\to\infty}\E^{\gamma_n}\left[c\big((B)_{\cdot\wedge T},T\big)\right]=\limsup_{n\to\infty} v(\mu_n).
			\end{align} 
		We first show that the sequence $(\gamma_n)$ is tight; for this it suffices to show that its respective projections onto $C_0(\R_+)$ and $\R_+$ are tight. The projections onto $C_0(\R_+)$ all coincide with the Wiener measure and are thus trivially tight. 
		On the other hand, $\gamma_n(T>r)=\mu_n([r,\infty))$. For any $\varepsilon>0$, using that $\mu_n$ converges to $\mu$, we may now choose an $r>0$, such that $\mu_n([r,\infty))<\varepsilon$, for all $n\in\Nb$; it follows that also the projections onto $\R_+$ are tight. 
		
		Let $\gamma$ be an accumulation point; by passing if necessary to a subsequence, we may assume that $\gamma_n\Rightarrow\gamma$.  
		Since $(\omega,t)\mapsto\omega$ is trivially continuous, it follows that $\gamma\big|_{\Omega}=\mathbb{W}$. Further, using the continuity of $(\omega,t)\mapsto t$ combined with the fact that $\mu_n\to_{\Wc_1}\mu$, we obtain for any $g\in C_b(\R_+)$, 
			\begin{align*}
				\int g(t)\gamma\big|_{\R_+}(\di t)
				=
				\E^{\gamma}\left[g(T)\right]
				=
				\lim_{n\uparrow\infty}\E^{\gamma_n}\left[g(T)\right]
				=
				\lim_{n\uparrow\infty}\int g(t)\mu_n(\di t)
				=
				\int g(t)\mu(\di t).
			\end{align*}
		Hence, $\gamma\in\mathsf{Cpl}(\mathbb{W},\mu)$. Finally, using Theorem 3.8 (3) in \cite{beiglboeck2016} we obtain $\gamma\in\mathsf{RST}(\mu)$. 
		
		 Next, according to Proposition 2.4 in \cite{jacod1981} (see also Lemma 4.2 in \cite{guo2016}), on the space $\mathsf{RST}(\mu)$, the weak convergence topology coincides with the stable convergence topology, which is the coarsest topology under which $\gamma\to\E^\gamma[\phi]$ is continuous for all bounded measurable functions $\phi:\Omega\times\R_+\to\R$ such that $t\mapsto\phi(\omega,t)$ is continuous for all $\omega\in\Omega$. Hence, by use of the assumption (and Portmanteau's lemma), we have
			\begin{align}\label{eq:proof_usc}
				v(\mu)
				\ge \E^{\gamma}\left[c\big((B)_{\cdot\wedge T},T\big)\right]
				\ge \limsup_{n\to\infty}\E^{\gamma_n}\left[c(\big(B)_{\cdot\wedge T},T\big)\right],
			\end{align}
		which combined with \eqref{eq:1} yields the required u.s.c. 
	\end{proof}

	Since the value function is concave\footnote{Indeed, let $\mu_1,\mu_2\in\Pc$, $\varepsilon>0$, and $\tau_1,\tau_2$ corresponding $\varepsilon$-optimal stopping times. For $\lambda\in(0,1)$, define $\tau=\ind_{\{U\le \lambda\}}\tau_1+\ind_{\{U>\lambda\}}\tau_2$ for some independent $U\sim U[0,1]$. Then, $\tau\in\mathcal{T}(\mu^\lambda)$ for $\mu^\lambda=\lambda\mu_1+(1-\lambda)\mu_2$, and working on the enlarged space $(\overline\Omega,\overline\Gc,(\overline\Gc_t),\overline\P)$ defined in the proof of Proposition \ref{prop:cont_right}, we obtain
			\begin{align*} 
				v(x,\mu^\lambda)
				\ge 
				\overline\E\left[\ind_{\{U\le \lambda\}}c\big(W_\cdot,\tau_1\big)+\ind_{\{U> \lambda\}}c\big(W_\cdot,\tau_2\big)\right]
				\ge 		
				\lambda v(x,\mu_1)+(1-\lambda)v(x,\mu_2)-2\varepsilon,
			\end{align*}
		which yields the concavity since $\varepsilon$ was arbitrary.}, when restricted to the set of probability measures with support on a finite number of fixed points -- which renders $v$ a function from $\R^n$ to $\R$ (cf. \eqref{eq:v_atomic_case}) -- it follows from Proposition \ref{lem:usc} that $v$ is continuous on the domain where it is finite. 
	The value function is also continuous whenever the problem admits a solution of barrier type; cf. \cite{manu,grass2016,loynes1970,root1969}. 
	In general, however, continuity of the value function is not a-priori clear. 
	Next we consider a class of pay-off functions for which we may nevertheless establish a certain 'right-continuity' of the value function. The result is of importance since it ensures that the value of the problem with a general constraint can be obtained as the limit of a sequence of approximating atomic problems, which are easier to address by use of numerical methods. 
	For this result we consider pay-off functions which satisfy the following additional condition:  
	
	\begin{assumption}\label{ass:special_form} Given a probability space supporting a BM $(W_t)_{t\ge 0}$, and a stopping time $\tau$, there exists a modulus of continuity $\varphi$ such that for any stopping time $\rho\ge\tau$, 
		\begin{align}\label{eq:ass_special_form}
			\left|\E\left[c(W_{\cdot\wedge\tau},\tau)-c(W_{\cdot\wedge\rho},\rho)\right]\right|
			\le \varphi\left(\E\left[|\tau-\rho|\right]\right).
		\end{align}
	\end{assumption} 
	
	\begin{example}\label{ex:ex_special_form}
		Assumption \ref{ass:special_form} holds for example in either of the following cases: 
		\begin{itemize}
		\item $c(\omega,t)=f(\omega_t)$ for some $2$-H\"older continuous function $f$, for then 
			$\E[|f(W_{\tau})-f(W_{\rho})|]
			\le c \E [|W_{\tau}-W_{\rho}|^2]
			= c \E[|\tau-\rho|]$;
		\item $c(\omega,t)=f(\omega^*_t)$ for some $2$-H\"older continuous function $f$, where $\omega^*_t=\sup_{\cdot\le t}\omega_\cdot$, since by Doob's inequality, $\E[|f(W^*_{\tau})-f(W^*_{\rho})|] \le c\E[|W^*_\tau-W^*_\rho|^2]\le 4c \E[|\tau-\rho|]$;
		\item the above two cases when $W$ is replaced by a local martingale with $\langle M\rangle_t\le t$ (i.e. evolving 'slower' than a Brownian motion);
		\item when $\left|c(\omega,t)-c(\omega,s)\right|\le \varphi(|t-s|)$ for a concave modulus of continuity $\varphi$.		
		\end{itemize}
	\end{example}
	
	We denote by $\Rc(\xi)$ the set of measures in $\Pc$ which may be obtained from $\xi$ by moving mass to the right; more precisely, if $\xi'\in\Rc(\xi)$ and $m(\cdot,\di y)$ is an disintegration in the first variable of an optimal coupling with respect to $\Wc_1$ between $\xi$ and $\xi'$, then $m(x,\di y)$ has support only on $[x,\infty)$.
	We then have the following result; we note that its proof bears resemblance to the proof of Lemma 3.1 in \cite{cox2015}.

	\begin{proposition}\label{prop:cont_right}
		Suppose that Assumption \ref{ass:special_form} holds. 		
		Let $\xi^n\in\Rc(\xi)$, $n\in\Nb$, be a sequence such that $\xi^n\to^{\Wc_1}\xi\in\Pc$. 
		Then, $v(\xi)=\lim_{n\to\infty}v(\xi^n)$. 
	\end{proposition}

		\begin{proof}[Proof of Proposition \ref{prop:cont_right}]
			Fix $\varepsilon>0$ and take $\tau\in\Tc(\xi)$ such that $\Ep{c(W,\tau)}\ge v(\xi)-\varepsilon$. 
			In turn, let $\Gamma^n$ such that $\Wc_1(\xi,\xi^n)=\iint |x-y|\Gamma^n(\dx,\di y)$, and let $m^n(x,\di y)$ the family of disintegrated kernels for which $\Gamma^n(\dx,\di y)=\xi(\dx)m^n(x,\di y)$. Let $U\sim U[0,1]$ independent of $\Gc$. Denoting by $M^{n,(-1)}(x,\cdot)$ the right-continuous inverse of $M^n(x,\cdot):=\int_x^\cdot m^n(x,\di s)$, we then define
			\begin{equation*}
				\tau^n:=M^{n,(-1)}(\tau,U).			
			\end{equation*}
		Since $\tau$ is $\Fc_{\tau}$-measurable and $\xi^n\in\mathcal{R}(\xi)$, on the enlarged space $(\overline\Omega,\overline\Gc,(\overline\Gc_t),\overline\P)=(\Omega\times[0,1],\Gc\otimes\mathcal{B}([0,1]),\Gc_t\otimes\mathcal{B}([0,1]),\P\times\mathcal{L})$, we then have that $\tau^n$ is a $(\overline\Gc_\cdot)$-stopping time with $\tau^n\sim\xi^n$. Moreover, 
 		\begin{align}\label{eq:proof_cont}
			\overline\E\left[|\tau-\tau^n|\right]
			&= \overline\E\left[\left|\tau-M^{n,(-1)}(\tau,U)\right|\right]\\
			&= \int_0^\infty|s-t|\xi(\ds)m^n(s,\di t)
			\;=\; \Wc_1\left(\xi,\xi^n\right).\nonumber
		\end{align} 
	We may thus choose $N\in\Nb$ such that $\varphi(\overline\E[\tau^n-\tau])\le \varepsilon$, for all $n\ge N$, where $\varphi$ is the modulus of continuity given in Assumption \ref{ass:special_form}; applying the latter assumption we then obtain
			\begin{align*}
				v(\xi^n)
				\ge \overline\E\left[c(W_\cdot,\tau^n)\right]
				\ge \overline\E\left[c(W_\cdot,\tau)\right] - \varepsilon
				\ge v(\xi)-2\varepsilon,
				\quad\textrm{for $n\ge N$}.  
			\end{align*}
		Since $\varepsilon>0$ was arbitrarily chosen, we obtain $\lim_{n\to\infty}v(\xi^n)\ge v(\xi)$, which combined with Proposition \ref{lem:usc} yields the result. 		
		\end{proof}

	\section{Measure-valued martingales and the DPP}
	
	\subsection{Measure-valued martingales and alternative problem formulation}
	
	The notion of measure-valued martingales was used in \cite{cox2015} to address a robust pricing problem where optimisation took place over a set of martingales satisfying a given marginal constraint. By reformulating the problem in terms of measure-valued martingales, the constraint was turned into an initial condition for an additional measure-valued state process, which enabled addressing the problem as a stochastic control problem. 
	The aim herein in to apply the same approach to the distribution-constrained optimal stopping problem.  
	To see that this is indeed natural, note that for a given $\Fb$-stopping time $\tau$, defining 
		\begin{align*} 
			\xi_t:=\Lc(\tau|\Fc_t),
		\end{align*} 
	we obtain a process $\xi_\cdot$ taking its values in the set of probability measures on $\R_+$, with $\xi_0=\mu$ and $\lim_{t\to\infty}\xi_t=\tau$; it is also a martingale in a sense to be made precise (cf. Definition \ref{def:MVM} below). In addition, the fact that $\tau$ is a stopping time implies that $\xi_\cdot$ has the following property: denoting by $\Pc^s = \{ \mu \in \Pc : \mu = \delta_y, y \in \Rp\}$, it holds that
		\begin{equation}\label{cond_st}
			\inf\{t\ge 0: \xi_t\in\Pc^s\}\le \xi_\infty;
		\end{equation}
	that is, either $\xi_t\in\Pc^s$ or $\supp \xi_t\subseteq [t,\infty)$ a.s. 
	Notably, there is a one-to-one correspondence between the family of stopping times $\tau$ and such measure-valued processes for the former may be recovered from the latter via $\tau=\arg\{t\ge 0:\xi_t=\delta_t\}$. 
	In principle, the distribution-constrained optimal stopping problem may thus, equivalently, be formulated as an optimisation problem over measure-valued martingales $\xi$ satisfying condition \eqref{cond_st} and the initial condition $\xi_0=\mu$.

	For our distribution-constrained stopping problem \eqref{problem}, we need however to consider a larger class of stopping times than the ones discussed above. The way we choose to formalise this, is that we take the \emph{conditional distribution of the disintegrated randomised stopping time} as our controlled variable. From the perspective of the problem formulation in \eqref{problem}, in the simplest case where $\Gb$ is generated by a Brownian motion initially enlarged by an independent uniformly distributed random variable, this corresponds to identifying a $\Gb$-stopping time with the process yielding its conditional law given the Brownian filtration alone; see however Remark \ref{rem:jumping_MVM} below for an alternative approach.
	More precisely, recall from \eqref{eq:problem_formulation_useful} that working on the space $(\Omega=C_0(\R_+),\Fc,\Wb)$, we may choose to view our problem as an optimisation problem over kernels $(\gamma_\omega)_{\omega\in C_0(\R_+)}$ corresponding to disintegrations of (adapted) couplings between $\Wb$ and $\mu$. 
	Viewing each such kernel as a $\Pc$-valued $\Fc$-measurable random variable satisfying the constraint that under $\Wb$ it averages to $\mu$, and aiming at including the latter constraint as an initial condition, we will identify each such kernel with the process yielding its conditional distribution. This allows us to view the problem as an optimisation problem over measure-valued martingales satisfying the constraint $\xi_0=\mu$ and a suitable adaptedness condition. 
	Specifically, we define the following class of optimisation objects: 
	
	\begin{definition}\label{def:MVM}
		Given a filtered probability space supporting an adapted process $(\xi_t)_{t \ge 0}$ with $\xi_t\in \Pc$, we say that
			\begin{itemize}
				\item[-] the process $\xi$ is a {\it measure-valued martingale} (MVM) if $\xi_\cdot(A)$ is a martingale, for any $A\in\Bc(\R_+)$; 
				\item[-] a MVM is \emph{continuous} if $t \mapsto \xi_t$ in continuous in the topology induced by $\Wc_1$, the first Wasserstein metric, for almost all $\omega\in\Omega$;
				\item[-] a MVM is \emph{adapted} if $\xi_t([0,s])=\xi_u([0,s])$ a.s., for all $s\le t\le u$. 
			\end{itemize}
	\end{definition}
	
	For any $A\in\Bc(\R_+)$, $\xi_\cdot(A)$ is trivially a uniformly integrable martingale with a well-defined limit $\xi_\infty(A)$; more pertinently, $\xi_\infty$ defines a probability measure, see Proposition 2.1 in \cite{horowitz1985}. 
	Further, note that the condition that $\xi_\cdot(A)$ is a martingale for any $A\in\Bc(\R_+)$, is equivalent to $\xi(f)$ being a martingale for any $f\in\mathcal{C}_b$; see Remark 2 in \cite{cox2015}.
	In particular, any MVM thus converges \as{} in the sense of weak convergence of measures to a limiting (random) measure $\xi_\infty$. 
	
	We denote by $\MVM(\mu)$ the set of continuous adapted MVMs $\xi$ with $\xi_0=\mu$.
	Our first claim then is that our original problem, indeed, admits the following equivalent formulation: 
	
	\begin{problem}\label{problem_mvm}
		On the given space $(\Omega=C_0(\R_+),\F,\Fb,\Wb)$, consider the problem of maximising
			\begin{equation}\label{eq:problem_mvm}
				\E\left[\int_0^\infty c\left(B_{\cdot\wedge s},s\right)\di A^\xi_s\right]
				\;\textrm{with}\;\;
				A^\xi_t:=\xi_t\left([0,t]\right),
				\;\textrm{over $\xi\in\MVM(\mu)$}.
			\end{equation}
	\end{problem} 
	
	\begin{lemma}\label{lem:proof_original_mvm}
		Problem \ref{problem_mvm} and the problem introduced in \eqref{problem} are equivalent. 
	\end{lemma} 
		
	\begin{proof} 
		Lemma 3.11 in \cite{beiglboeck2016} and Theorem VI 65 in \cite{dellacherie} (cf. Theorem 3.6 in \cite{beiglboeck2016}) combined implies that Problem \eqref{problem} is equivalent to optimising $\E\left[\int_0^\infty c(B_{\cdot\wedge s},s)\di A^\gamma_s\right]$, with $A^\gamma_t(\omega):=\gamma_\omega([0,t])$, over kernels $(\gamma_\omega)_{\omega\in C_0(\R_+)}$ corresponding to disintegrations in the first variable of measures $\gamma\in\RST(\mu)$; cf. \eqref{eq:problem_formulation_useful}. 
		Viewing each such kernel $(\gamma_\omega)$ as a $\Pc$-valued $\F$-measurable random variable, we define the associated process $(\xi_t)_{t\ge 0}$ by $\xi_t(A):=\E[\gamma_\cdot(A)|\Fc_t]$, for $A\in\Bc(\R_+)$; 
		since $\mu\in\Pc$ the thus defined process $\xi$ is in $\Pc$, and it is a measure-valued martingale, see e.g. Lemma 2.12 in \cite{cox2015} or Theorem 1.3 in \cite{horowitz1985}. 
		Notably, the limit $\xi_\infty:=\lim_{t\to\infty}\xi_t$ exists and $\xi_\infty(\di t;\omega)=\gamma_\omega(\di t)$, for $\Wb$-a.a. $\omega\in\Omega$. 
		In consequence, the objective functions in \eqref{eq:problem_formulation_useful} and \eqref{eq:problem_mvm} evaluated, respectively, with respect to $(\gamma_\omega)$ and $\xi$, coincide. 
		
		Next, note that since $\gamma\in\RST(\mu)$, or, equivalently, the kernels $\gamma_\omega$ average to $\mu$ under $\Wb$, and since $\Fc_0$ is trivial, we obtain $\xi_0=\mu$. 
		Further, according to Remark 4 in \cite{cox2015}, since the filtration $\Fb$ satisfies the usual conditions, any MVM in $\Pc$ admits a version which is right-continuous in the sense that $\xi_\cdot(f)$ is right-continuous for any $1$-Lipschitz function $f$; w.l.o.g., we choose this version. Moreover, since the filtration $\Fb$ is generated by a BM, it follows from the martingale representation theorem that each process $\xi_\cdot(f)$ is in fact continuous, and thus $\xi$ is continuous in the sense of Definition \ref{def:MVM}.		
		Since for any $\gamma\in\RST(\mu)$, $\gamma_\cdot([0,t])$ is $\Fc_t$-measurable, and since $\xi_\infty(\omega;[0,t])=\gamma_\omega([0,t])$ a.s., we also obtain that $\xi_t$ is adapted in the sense of Definition \ref{def:MVM}. 
		In consequence, the thus defined process $\xi$ is indeed a continuous adapted MVM with $\xi_0=\mu$. 
		
		Conversely, for any $\xi\in\MVM(\mu)$, there exists a measure $\gamma\in\RST(\mu)$ such that its disintegrated kernel $(\gamma_\omega)$ satisfies $\gamma_\omega(\dt)=\xi_\infty(\dt;\omega)$ a.s.; we easily conclude.  
	\end{proof}

	\begin{remark}\label{rem:equivalence_weak_MVM}
		We choose in Problem \ref{problem_mvm} to work on the fixed space $(\Omega,\F,\Fb,\Wb)$ but note that this choice is in fact arbitrary and that we could have chosen any filtered probability space $(\Omega,\Gc,\Gb,\P)$ (satisfying the usual conditions) with $\Gc_0$ trivial and supporting a BM $W$.
		Indeed, denote by $(\overline\Omega,\overline\Gc,\overline\Gb,\overline\P)$ the probability space obtained by initially enlarging such a space by an independent uniformly distributed random variable (cf. the proof of Proposition \ref{prop:cont_right}). For any given adapted MVM we may then construct a stopping time $\tau$ in the latter space such that the value of the respective objective functions coincide in that $\overline \E[c(W,\tau)]=\E[\int c(W,s)\di A^\xi_s]$. 
		Conversely, given a stopping time $\tau$ in $(\overline\Omega,\overline\Gc,\overline\Gb,\overline\P)$, by use of the same arguments as used in the proof of Lemma \ref{lem:proof_original_mvm}, we see that defining $\xi_t:=\Lc(\tau|\Gc_t)$ yields an adapted MVM in $(\Omega,\Gc,\Gb,\P)$ for which the corresponding objective functions coincide; notably $\xi_\cdot$ admits a \cadlag{} version and satisfies $\xi_0=\mu$ since $\Gc_0$ is trivial.   
	
	The independence of the specific choice of the probability space for problem \eqref{problem} (cf. Section \ref{sec:preliminaries}) is thus inherited also by the MVM formulation of the problem. 
		In particular, denoting by $\mathcal{K}(\mu)$ the set of all terms $\kappa=(\Omega^\kappa,\Fc^\kappa,\Fb^\kappa,\P^\kappa,W^\kappa,\xi^\kappa)$ such that $(\Omega^\kappa,\Fc^\kappa,\Fb^\kappa,\P^\kappa)$ is a filtered probability space with $\Gc^\kappa_0$ trivial and in which $W^\kappa$ is a BM and $\xi^\kappa$ an adapted MVM with $\xi^\kappa_0=\mu$, we have 
		\begin{align}\label{eq:weak_kappa_form}
			v(\mu)=\sup_{\mathcal{K}(\mu)}\E^\kappa\left[\int c\left(W^\kappa_{\cdot\wedge s},s\right)\di A^{\xi^\kappa}_s\right].
		\end{align} 
	\end{remark}
	
	\begin{remark}\label{rem:jumping_MVM}
		Viewed from the perspective of problem \eqref{problem}, for the simplest case in which $\Gb$ is the filtration generated by a Brownian motion initially enlarged by an independent uniformly distributed random variable, the MVMs considered in Problem \ref{problem_mvm} correspond to $\xi_t:=\mathcal{L}(\tau|\Fc_t)$ where $(\Fc_t)$ denotes the filtration generated by the Brownian motion alone; alternatively, one may consider objects of the form $\eta_t:=\mathcal{L}(\tau|\Gc_t)$. The latter also defines MVMs, which in addition terminate in a singular measure in that $\eta_\infty(\omega)=\delta_{\tau(\omega)}\in\Pc^s$ a.s. Typically, however, we then do not have $\eta_0=\mu$. Nevertheless, without destroying the martingale property, one may extend $\eta$ onto $(-\varepsilon,\infty)$ by defining $\eta_t=\mu$ for $t\in(-\varepsilon,0)$; $\eta$ will then generally possess a jump at time $t=0$ following the realisation of the $\Gc_0$-measurable random variable. Naturally, the results herein may also be formulated using such MVMs; we consider however the present formulation to be most natural for our purposes and do not pursue those details any further; see, however, Corollary \ref{cor:strong_dpp} below. 
	\end{remark}

	\subsection{The Dynamic Programming Principle}
	
	The aim in this section it to establish the dynamic programming principle for the distribution-constrained optimal stopping problem in its equivalent form \eqref{eq:problem_mvm}. To this end we introduce the associated value function. We continue to work on the space $(\Omega=C_0(\R_+),\Fc,\Fb,\Wb)$, although in light of Remark \ref{rem:equivalence_weak_MVM}, we note that this is a somewhat arbitrary choice. 
	Specifically, we define $v:\R_+\times C_0(\R_+)\times\Pc\to\R$ by 
		\begin{align}\label{eq:value-function}
			v(t,\vec{w},\xi)
			:=\sup_{\xi_\cdot\in\MVM^t(\xi)}\E\left[\int_t^\infty c\left(W^{t,\vec{w}}_{\cdot\wedge u},u\right)\di A^\xi_u\right],
		\end{align}
	where $\MVM^t(\xi)$ denotes the set of continuous adapted MVMs with $\xi_t=\xi$ a.s., and $W^{t,\vec{w}}_\cdot$ denotes the solution to the SDE $\di W^{t,\vec{w}}_s=\di B_s$, for $s\in[t,\infty)$, with initial condition $W^{t,\vec{w}}_s=\vec{w}_s$ a.s., for $s\in[0,t]$.

	Our main result is the following; we use the convention $\E[\Xi]=\E[\Xi^+]-\E[\Xi^-]$ with $\infty-\infty=-\infty$, for any random variable $\Xi$: 
	
	\begin{theorem}[Dynamic Programming Principle]\label{thm:dpp}
	For all $(t,\vec{w},\xi)\in\R_+\times C_0(\R_+)\times\Pc$, and any $\Fb$-stopping time $\theta$ with values in $(t,\infty)$, it holds that
		\begin{align}\label{eq:dpp}
			v(t,\vec{w},\xi)
			=\sup_{\xi_\cdot\in\MVM^t(\xi)}
			\E\left[
			\int_t^\theta c\left(W^{t,\vec{w}}_{\cdot\wedge u},u\right)\di A^\xi_u
			+ v\left(\theta,W^{t,\vec{w}}_{\cdot\wedge\theta},\xi_\theta\right)\right].
		\end{align}		 
	\end{theorem}

	\begin{remark}\label{rem:thm_alt_scaling}
		Integrals with respect to $\di A^\xi_\cdot$ are naturally understood in the sense of (path-wise) Lebesgue–Stieltjes integration. Hence, if $\xi\in\Pc$ has an atom at time $t$, it will not contribute to the value function $v(t,\vec{x},\xi)$. Similarly, in the formulation on the DPP \eqref{eq:dpp}, if a given MVM $\xi_\cdot\in\MVM(\xi)$ has an atom at time $\theta$, that atom will contribute to the objective function via the integral from $t$ to $\theta$ rather than via the value function evaluated at time $\theta$.		
		More specifically, the value function only depends on $\vec{w}$ up to time $t$ and $\xi$ on $(t,\infty)$ in the sense that $v(t,\vec{w},\xi)=v(t,\vec{w}_{\cdot\wedge t},\tilde\xi^t)$, where $\tilde\xi^t(A)=\xi(A\cap(t,\infty))$, $A\in\Bc(\R_+)$. For this reason we introduce the notation $\xi^{|_t}$ for the (re-weighted) restriction of $\xi$ to $(t,\infty)$: $\xi^{|_t}(A):=\frac{\xi(A\cap(t,\infty))}{\xi((t,\infty))}$, $A\in\Bc(\R_+)$. 
	It follows from the definition of $v$ that $v(t,\vec{w},\xi)=\xi((t,\infty))v(t,x,\xi^{|_t})$; hence we also obtain the following version of the DPP: 	
	\begin{align*}
			v(t,\vec{w},\xi)
			=\sup_{\xi_\cdot\in\MVM^t(\xi)}
			\E\left[
			\int_t^\theta c\left(W^{t,\vec{w}}_{\cdot\wedge u},u\right)\di A^\xi_u
			+ \xi_\theta\big((\theta,\infty)\big)v\left(\theta,W^{t,\vec{w}}_{\cdot\wedge\theta},\xi_\theta^{|_\theta}\right)\right].
		\end{align*}
	\end{remark}

	Before giving the proof of Theorem \ref{thm:dpp} (in Section \ref{sec:proof_dpp} below) we comment on two particular cases. 
	First, although we consider the (weak) problem formulation studied herein to be the natural one (in particular since it always admits a solution), we note that whenever problem \eqref{problem} admits a 'strong' solution\footnote{
	For example, this is the case when \eqref{problem} admits a barrier solution, which according to \cite{manu} happens whenever $c(\omega,t)=f\circ\omega(t)$, for some $f:\R\to\R$ with $f'''>0$. Notably, for the special case when there exists a barrier solution, the DPP principle is an immediate consequence thereof.
	}
	(here meaning an optimal stopping time adapted to the filtration generated by the Brownian motion alone), Theorem \ref{thm:dpp} implies a stronger result. 
	In order to specify the DPP in this case, recall from Definition \ref{def:MVM} and the subsequent discussion that there exists a limiting (random) measure $\xi_\infty\in\Pc$ \as{} such that $\xi_t \to \xi_\infty$ \as{} as $t\to\infty$, where the convergence is in the sense of weak convergence of measures; further, recall the notation $\Pc^s = \{ \mu \in \Pc : \mu = \delta_y, y \in \Rp\}$. We then define as follows:  	
		
	\begin{definition}\label{def:TMVM}
		A MVM $(\xi_t)$ is {\it terminating} if $\xi_\infty\in \Pc^s$ \as{}; it is {\it finitely terminating} if $\tau^\xi:= \inf \{t \ge 0: \xi_t \in \Pc^s\}$ is almost surely finite.  
	\end{definition}	

	Notably, a terminating MVM satisfies the adaptedness property, if and only if, condition \eqref{cond_st} holds; in consequence, any adapted terminating MVM is, in fact, finitely terminating and we have $\tau^\xi=\arg \{t \ge 0: \xi_t = \delta_t\}$. 
	We denote the set of continuous adapted and finitely terminating MVMs by $\MVM_{\mathrm{term}}$, and define $\MVM_{\mathrm{term}}^t(\xi)$ analogously to above.  
	
	\begin{corollary}[DPP: strong formulation]\label{cor:strong_dpp}
		Suppose that restricting in problem \eqref{problem} to stopping times measurable with respect to the filtration generated by the Brownian motion alone, does not affect the value of the problem. 
		Then, for any $\Fb$-stopping time $\theta$ with values in $(t,\infty)$,
		\begin{align}\label{eq:dpp_cor}
			v(t,\vec{w},\xi) 
			=\sup_{\substack{\xi_\cdot\in\MVM^t_{\mathrm{term}}(\xi)}}
			\E\Big[
			c\left(W^{t,\vec{w}}_{\cdot\wedge \tau^\xi},\tau^\xi\right)\ind_{\left\{\tau^\xi\le\theta\right\}} 
			+v\left(\theta,W^{t,\vec{w}}_{\cdot\wedge\theta},\xi_\theta\right) 
			\ind_{\left\{\tau^\xi>\theta\right\}}\Big].
		\end{align}
	\end{corollary}

	Second, we consider the case when $\mu$ has support on a finite number of (fixed) atoms $0<t_1<...<t_r$. For $k\in\{1,...,r\}$ and $t\in[t_{r-k},t_{r-k+1})$, let $v_t:C_0(\R_+)\times\Delta_k\to\R$ be given by
		\begin{equation}\label{eq:v_atomic_case}
			v_t(\vec{w},y_{(r-k+1):r}):=v(t,\vec{w},\xi),
			\;\;\textrm{with}\;\; \xi = \sum_{i=r-k+1}^ry_i\delta_{t_i},		
		\end{equation}
	where we use the notation $y_{(r-k+1):r}$ for the vector $(y_{r-k+1},...,y_r)$ and similarly for vector-valued stochastic processes. 
	We then have the following corollary:  
	
	\begin{corollary}[DPP: atomic constraint]\label{cor:dpp}
		Suppose that $\mu$ has support on the finite number of atoms $0<t_1<...<t_r$.
		Then, for $k\in\{1,...,r\}$, with $\mathcal{M}(\Delta_k)$ denoting the set of martingales taking values in $\Delta_k$, we have that
		\begin{align}\label{eq:dpp_cor_2}
			v_{t_{r-k-1}}\left(\vec{w},y_{(r-k):r}\right)
			=&\sup_{\substack{Y\in\mathcal{M}(\Delta_{k+1})\\Y_{t_{r-k-1}}=y_{(r-k):r}}}
			\E\Bigg[Y^{r-k}_{t_{r-k}}c\left(W^{t,\vec{w}}_{\cdot\wedge t_{r-k}},t_{r-k}\right)\\
			&\phantom{........}+\left(1-Y^{r-k}_{t_{r-k}}\right)v_{t_{r-k}}\left(W^{t,\vec{w}}_{\cdot\wedge t_{r-k}},
			\frac{Y^{(r-k+1):r}_{t_{r-k}}}{1-Y^{r-k}_{t_{r-k}}}\right)\Bigg].\nonumber
		\end{align} 
	\end{corollary}
	
	\begin{remark}[Relation to Bayraktar and Miller \cite{bayraktar2016}]\label{rem:bayraktar_new}
		In \cite{bayraktar2016}, the authors consider cost functions of the type $c(\omega_{\cdot\wedge s},s)=f\circ\omega(s)$ with $f:\R\to\R$ Lipschitz, and restrict to atomic marginal constraints. Moreover, they consider the strong formulation where the filtration is the one generated by the Brownian motion alone. 
		Their main result (Theorem 1) then establishes \eqref{eq:dpp_cor_2} for their value function $v$; notably they first derive the strong version with the additional condition $Y^{r-k}_{t_{r-k}}\in\{0,1\}$ a.s. (cf. \eqref{eq:dpp_cor}) and then relax this condition. 
		Combining our Corollary \ref{cor:dpp} with their result, we see that for the cost functions and constraints considered in \cite{bayraktar2016}, the value function when restricting to Brownian stopping times coincides with the value function for the weak problem formulation. Hence, a-posteriori, we recover their result as a special case of ours.  
	\end{remark}

	\subsection{Proof of the DPP}\label{sec:proof_dpp}
	
	Since our objective function is given in Lagrange form (with a reward function integrated over time), we first introduce an additional state-variable (governing its accumulated value) in order to transform it onto Mayer form. Specifically, given $(t,\vec{w},y)\in\R_+\times C_0(\R_+)\times\R$ and $\xi_\cdot\in\MVM$, we define the process 
		\begin{align}\label{eq:proof_y}
			Y^{t,\vec{w},y,\xi}_u:=y+\int_t^u c\left(W^{t,\vec{w}}_{\cdot\wedge s},s\right)\di A^\xi_s;
		\end{align}
	we note that it admits a well-defined limit which we denote by $Y^{t,\vec{w},y,\xi}_\infty$. For $(t,\vec{w},y,\xi)\in\R_+\times C_0(\R_+)\times\R\times\Pc$, we introduce the value function
		\begin{align}\label{eq:value_function_z}
			\tilde v(t,\vec{w},y,\xi):=\sup_{\xi_\cdot\in\MVM^t(\xi)}\E[Y^{t,\vec{w},y,\xi}_\infty].
		\end{align} 
	We then have that Theorem \ref{thm:dpp} follows if, for any $\Fb$-stopping time $\theta$ with values in $(t,\infty)$, it holds that
		\begin{align}\label{eq:dpp_z}
			\tilde v(t,\vec{w},y,\xi)=\sup_{\xi_\cdot\in\MVM^t(\xi)}
			\E\left[\tilde v\left(\theta,W^{t,\vec{w}}_{\cdot\wedge \theta},Y^{t,\vec{w},y,\xi}_u,\xi_\theta\right)\right]. 
		\end{align}
	
	Our aim is to prove this result by considering its weak formulation on the associated canonical path space; the canonical framework has previously successfully been used for the study of stochastic control problems in \cite{karoui2013_1,karoui2013_2}, see also \cite{zitkovic2014} and \cite{neufeld2013,nutz2013}, and was used to deduce the DPP also in \cite{cox2015}. 
	
	We denote by $\mathbb{D}$ the set of c\`adl\`ag paths on $[0,\infty)$ taking values in $E:=\R\times\R\times\Pc$, where we equip $\Pc$ with the topology induced by the $\mathcal{W}_1$-metric and $E$ with the product topology; this renders $E$ a Polish space, and using the Skorokhod topology on $\mathbb{D}$ it is a Polish space too\footnote{
	We note that the argument given in this section also works if equipping the set of probability measures on $\R_+$, where $\xi_\cdot$ takes its values, with the weak topology. Hence, in contrast to the robust pricing situation considered in \cite{cox2015}, here the assumption that $\mu$ has a finite first moment could be weakened. We find however this assumption rather natural and in order to stay consistent with the remainder of the paper we keep it throughout.
	}. 
	A generic path in $\mathbb{D}$ is denoted by $\omega$ and we use $X=(W,Y,\xi)$ for the co-ordinate process: $X_t(\omega)=(W_t,Y_t,\xi_t)(\omega)=\omega(t)$; we will also use the notation $\vec{x}$ to refer to a path in $\Db$. 
	 We let $\mathbb{F}^0=\{\Fc^0_t\}_{t\in[0,\infty)}$ denote the filtration generated by the co-ordinate process $X$.
	The set of all probability measures on $\mathcal{B}(\mathbb{D})$ is denoted by $\Pk$; for $(t,\vec{x})\in\R_+\times\Db$, we denote by $\Pk_{t,\vec{x}}$ the set of probability measures $\P$ in $\Pk$ for which
\begin{enumerate}
\item\label{item:1} $X_s=\vec{x}_s$, $0\le s\le t$, $\P$-a.s.;
\item\label{item:3} $(W_u-W_t)_{u\ge t}$ is a $(\Fb^0,\P)$-Brownian motion; 
\item\label{item:2} $(\xi_u)_{u\ge t}$ is an adapted measure-valued $(\Fb^0,\P)$-martingale,
\item\label{item:4} $Y_u=y+\int_t^u c\left(W_{\cdot\wedge s},s\right)\di A^\xi_s$, for $u\ge t$, $\P$-a.s., where $A^\xi_s=\xi_s([0,s])$. 
\end{enumerate}
	Notably, $\Pk_{t,\vec{x}}$ only depends on $\vec{x}$ via $(W_{\cdot\wedge t},Y_t,\xi_t)(\vec{x})$; for ease of notation we keep the former notation. 
	We note that there exists a measurable functional $G:\mathbb{D}\to \R_+$ such that $G(\omega)=\lim_{t\to\infty}Y_t(\omega)$ whenever the limit exists (recall that $c(\cdot,\cdot)$ is measurable and see Lemma 3.12 in \cite{zitkovic2014}); in particular, for any $\P\in\Pk_{t,\vec{x}}$, with $(t,\vec{x})\in\R_+\times\Db$, we have $G=\lim_{t\to\infty}Y_t$, $\P$-a.s.  
	We then define 
	\begin{equation}\label{eq:vdefn}
 		v(t,\vec{x}):=\sup_{\P\in\Pk_{t,\vec{x}}}\E^\P\left[G\right]. 
	\end{equation}

	\begin{lemma}\label{lem:equiv_weak_mvm}
		The value functions defined in \eqref{eq:value_function_z} and \eqref{eq:vdefn} coincide. 
	\end{lemma}
	\begin{proof}
		W.l.o.g., let $t=y=0$, $\xi=\mu$ and $\vec{x}=(0,\vec{w},0,\mu)$.
		First, let $\mathcal{K}(\mu)$ denote the set of all terms $\kappa=(\Omega^\kappa,\Fc^\kappa,\Fb^\kappa,\P^\kappa,W^\kappa,\xi^\kappa)$ such that $(\Omega^\kappa,\Fc^\kappa,\Fb^\kappa,\P^\kappa)$ is a filtered probability space in which $W^\kappa$ is a BM and $\xi^\kappa$ an adapted MVM with $\xi^\kappa_0=\mu$; we denote the associated process defined via \eqref{eq:proof_y} by $Y^\kappa$. 
		We then have $v(0,\vec{x})=\sup_{\mathcal{K}(\mu)}\E[Y^\kappa_\infty]$. Indeed, any multiple $\kappa\in\mathcal{K}(\mu)$ induces a term in $\Pk_{0,\vec{x}}$. Conversely, any probability measure $\P\in\Pk_{0,\vec{x}}$ together with the space $(\Db,\Bc(\Db),\Fb^\P)$ and the canonical process $(W,\xi)$ produces such a multiple, this since the properties (i)-(iv) hold also with respect to the augmented filtration $\Fb^\P$. 
		Next, according to Lemma \ref{lem:proof_original_mvm} and Remark \ref{rem:equivalence_weak_MVM} (cf. \eqref{eq:weak_kappa_form}), we have that $\sup_{\mathcal{K}(\mu)}\E[Y^\kappa_\infty]=v(0,\vec{w},\xi)$. Since $v(0,\vec{w},\xi)=\tilde v(0,\vec{w},0,\xi)$, the result follows. 
	\end{proof}
	
	In order to establish the DPP for problem \eqref{eq:vdefn} we first establish two lemmas.

	\begin{lemma}\label{lem:analytic}
		The set $\Gamma:=\{(t,\vec{x},\P):(t,\vec{x})\in\R_+\times\Db,\; \P\in\Pk_{t,\vec{x}}\}$ is analytic. 
	\end{lemma}
	\begin{proof}
		For $0\le q<r$, $\Xi\in C_b(\Db,\F_q)$, $\varphi\in C_b^2(\R)$ and $A\in\Bc(\R_+)$, we consider the following subsets of $\R_+\times\Db\times\Pk$: 
		\begin{enumerate}
		\item[a)] $\left\{(t,\vec{x},\P): \P\left(X_{r\wedge t}=\vec{x}(r\wedge t)\right)=1\right\}$;
		\item[b)] $\big\{(t,\vec{x},\P): \E^\P\big[\Xi\big(\varphi(W_{r\vee t})-\varphi(W_{q\vee t})-\frac{1}{2}\int_{q\vee t}^{r\vee t}\varphi''(W_u)\di u\big)\big]=0\big\}$;
		\item[c)] $\left\{(t,\vec{x},\P): \E^\P\left[\Xi\left(\xi_{r\vee t}(A)-\xi_{q\vee t}(A)\right)\right]=0\right\}$;
		\item[d)] $\left\{(t,\vec{x},\P): \P\left(\xi_r(A\cap[0,r])=\xi_q(A\cap[0,r]\right)=1\right\}$;
		\item[e)] $\big\{(t,\vec{x},\P): \P\big(Y_{r\vee t}=y+\int_t^{r\vee t} c\left(W_{\cdot\wedge u},u\right)\di A^\xi_u\big)=1\big\}$.
		\end{enumerate}	
		The above sets are all Borel measurable. 
		Furthermore, $\Gamma$ is the intersection of the above sets when $q<r$ are allowed to vary among all the rational numbers in $\R_+$; $\Xi$ and $\varphi$ among countable dense subsets of, respectively, $C_b(\Db,\F_q)$ and $C_b^2(\R)$; and $A$ among a countable collection of sets generating $\Bc(\R_+)$.	
		Indeed, for the adaptedness property of $\xi$, note that for any $t<s<u$ it holds that $\xi_u(A\cap[0,s])=\xi_r(A\cap[0,s])$ for any rational number $r\in(s,u)$, and thus
			\begin{align*}
				\xi_u(A\cap[0,s])=\lim_{r\searrow s, r\in\mathbb{Q}}\xi_r(A\cap[0,s])=\xi_s(A\cap[0,s]),
				\quad \mathrm{a.s.},
			\end{align*}
		where the right-continuity of $\xi$ was used in the last equality; the remaining properties are immediate. 
		In consequence, also $\Gamma$ is Borel and thus analytic. 
	\end{proof}
	
	A map $\Q:\Db\times \mathcal{B}(\mathbb{D})\to[0,1]$ is called a (universally) measurable kernel if  i) $\Q(\omega,\cdot)\in\Pk$ for all $\omega\in\Db$, and ii) $\Db\ni \omega\to\Q(\omega,A)$ is (universally) measurable for all $A\in\mathcal{B}(\mathbb{D})$. 
	Recall that the universal $\sigma$-algebra is the intersection of the completions of the Borel $\sigma$-algebra over all probability measures on the space, and that universally measurable functions are integrable with respect to any such probability measure; we will use the superscript $U$ to refer to the universal completion.  
	We write $\Q_\omega$ for the probability measure $\Q(\omega,\cdot)$ and interpret $\Q$ as a (universally) measurable map $\Db\to\Pk$.	
	Further, given a random time $\theta$ and two paths $\omega,\omega'\in\mathbb{D}$ with $X_{\theta(\omega)}(\omega)=X_{\theta(\omega)}(\omega')$, we define the concatenation $\omega\otimes_\theta\omega'$ to be the element of $\mathbb{D}$ given by 
		\begin{equation*}
  			X_t(\omega\otimes_\theta\omega')
  			=
	  		\ind_{\{t<\theta(\omega)\}} X_t(\omega)
  			+
	  		\ind_{\{t\ge \theta(\omega)\}} X_{t}(\omega').
		\end{equation*}
	Given a probability measure $\P\in\Pk$ and a universally measurable kernel $\Q_\cdot$, we then define the concatenation $\P\otimes_\theta\Q_\cdot$ as the probability measure in $\Pk$ given by 
		\begin{equation*}
  			(\P\otimes_\theta\Q_\cdot)(A)=
  			\iint\ind_A(\omega\otimes_\theta\omega')\Q_{\omega}(\di\omega')\P(\di\omega), 
	  		\quad A\in\mathcal{B}(\mathbb{D}). 
		\end{equation*}

	\begin{lemma}\label{lem:cond_forward_backward} 
		Given $(t,\vec{x})\in\R_+\times\Db$, let $\P\in\Pk_{t,\vec{x}}$ and $\theta$ an $\Fb^0$-stopping time with values in $(t,\infty)$. Then,
		i) there is a family of r.c.p.d. $(\P_\omega)_{\omega\in\Db}$ of $\P$ w.r.t. $\F^0_\theta$ such that $\P_\omega\in\Pk_{\theta(\omega),\omega}$ for $\P$-almost every $\omega\in\Db$;
		ii) given $(\Q_\omega)_{\omega\in\Db}$ such that $\omega\mapsto\Q_\omega$ is $\F^U_\theta$-measurable and $\Q_\omega\in\Pk_{\theta(\omega),\omega}$ for $\P$-almost every $\omega\in\Db$, $\P\otimes_\theta\Q_\cdot\in\Pk_{t,\vec{x}}$. 	
	\end{lemma}
	
	\begin{proof}
		i) Since $\F^0_\theta$ is countably generated, there is a family of r.c.p.d. $(\P_\omega)_{\omega\in\Db}$ of $\P$ such that $\P_\omega(X_s=\omega_s,\;s\in[0,\theta(\omega)])=1$. 
		Recall from Theorem 1.2.10 in \cite{stroock} that for a given martingale $M$ on $[t,\infty)$, there exists a $\P$-null set $N\in\F^0_\theta$ such that $M$ is a $\P_\omega$-martingale after $\theta(\omega)$ for every $\omega\not\in N$. 
		Since the fact that $W$ is a BM and $\xi$ a MVM are characterised by the martingale property of $M^\varphi_\cdot=\varphi(W_{\cdot\wedge t})-\varphi(W_{t})-\frac{1}{2}\int_{t}^{\cdot}\varphi''(W_u)\di u$ and $\xi_{\cdot\wedge t}(A)$, where $\varphi$ runs through a countable dense subset of $C_b^2(\R)$ and $A$ through a countable algebra generating $\Bc(\R_+)$, we conclude that there exists a $\P$-null set $\tilde N\in\Fc^0_\theta$ such that for every $\omega\not\in\tilde N$, $W$ is a BM and $\xi$ a MVM on $[\theta(\omega),\infty)$ under $\P_\omega$. 
		Further, recall that the adaptedness property of $\xi$ and the property iv) are, respectively, characterised by the fact that $\xi_s(A\cap[0,s])=\xi_u(A\cap[0,s])$ a.s. for all $t\le s\le u$, and  $Y_u=\int_t^uf(X_u)\di A^\xi_u$ a.s. for all $u\ge t$. Since any $\P$-null set is a $\P_\omega$-null set for $\P$-almost all $\omega\in\Omega$, it follows that $\P_\omega\in\Pk_{\theta(\omega),\omega}$ for $\P$-almost every $\omega\in\Db$.
		
		ii) Given $(\Q_\omega)_{\omega\in\Db}$ as specified in the statement of the lemma, applying once again Theorem 1.2.10 in \cite{stroock} (cf. also the proof of Lemma 3.3 in \cite{karoui2013_2}), we obtain that $M^\varphi_{\cdot\wedge t}$ and $\xi_{\cdot\wedge t}(A)$ as defined above are indeed martingales also under $\P\otimes_\theta\Q_\cdot$. Moreover, since a set which is a null set under $\P_\omega$ for $\P$-almost all $\omega\in\Db$, is a null set also under $\P\otimes_\theta\Q_\cdot$, we conclude that also the adaptedness property and property iv) holds under the latter measure; hence, $\P\otimes_\theta\Q_\cdot\in\Pk_{t,\vec{x}}$.  		
	\end{proof}

	 We are now ready to prove the DPP for the weak problem formulation. We note that given the analyticity and stability provided by Lemmas \ref{lem:analytic} and \ref{lem:cond_forward_backward}, the result follows by rather standard arguments; see e.g. Theorem 2.3 in \cite{nutz2013} or Theorem 2.1 in \cite{karoui2013_2}; for completeness we provide the proof.
	 Recall that we use the convention $\E[\Xi]=\E[\Xi^+]-\E[\Xi^-]$ with $\infty-\infty=-\infty$, for any random variable $\Xi$: 

	\begin{theorem}
		For all $(t,\vec{x})\in\R_+\times\Db$, and any $\Fb^0$-stopping time $\theta$ with values in $(t,\infty)$, it holds that
  			\begin{equation*}
    			v(t,\vec{x})=\sup_{\P\in\Pk_{t,\vec{x}}}\E^\P\left[v(\theta,X_{\cdot\wedge\theta})\right].
  			\end{equation*}
	\end{theorem}

	\begin{proof}
	Given $(t,\vec{x})\in\R_+\times\Db$ and an $\Fb^0$-stopping time $\theta$ with values in $(t,\infty)$, let $\P\in\Pk_{t,\vec{x}}$ and recall that according to Lemma \ref{lem:cond_forward_backward}, there exists a family of r.c.p.d. of $\P$ given $\F^0_\theta$, say $(\P_\omega)_{\omega\in\Db}$, such that $\P_\omega\in\Pk_{\theta(\omega),\omega}$ for $\P$-almost every $\omega\in\Db$.  	
	By the properties of the r.c.p.d., we thus obtain	
	\begin{eqnarray*}
    \E^\P[G]
    =\E^\P\left[\E^{\P_\omega}[G]\right]
    \le\E^\P\left[v(\theta,X_{\cdot\wedge\theta})\right],
  	\end{eqnarray*}	
  	where it was used that $\P_\omega\in\Pk_{\theta(\omega),\omega}$ implies that $\E^{\P_\omega}[G]\le v(\theta(\omega),\omega)$, for $\P$-almost every $\omega\in\Db$. 
 	It follows that $v(t,\vec{x})\le\sup_{\P\in\Pk_{t,\vec{x}}}\E^\P\left[v(\theta,X_{\cdot\wedge\theta})\right]$ since  $\P\in\Pk_{t,\vec{x}}$ was arbitrarily chosen. 

	Next, recall that according to Lemma \ref{lem:analytic}, the set $\Gamma=\{(t,\vec{x},\P):\P\in\Pk_{t,\vec{x}}\}$ is analytic. 
	In particular, this implies that $v$ is upper semi-analytic. Indeed, the level sets of $v$ are given by $\{(t,\vec{x})\in\R_+\times\Db:v(t,\vec{x})>c\}=\mathrm{proj}_{\R_+\times\Db}\mathrm{L}^{>c}$, where 
		\begin{align*}
			\mathrm{L}^{>c}:=
			\big\{(t,\vec{x},\P)\in\R_+\times\Db\times\Pk:\E^\P[G]>c\big\}
			\cap
			\Gamma, 
			\;\; c\in\R, 
		\end{align*}
	which are analytic since $\P\mapsto\E^\P[G]$ is Borel measurable. 
	Given $\varepsilon>0$, it follows, in turn, that also the following set is analytic: 
  	\begin{align*}
  		\Gamma_\varepsilon=\{(t,\vec{x},\P):\E^\P[G]\ge v^\varepsilon(t,\vec{x})\}\cap \Gamma,
	\end{align*} 
	where $v^\varepsilon(t,\vec{x}):=(v(t,\vec{x})-\varepsilon)\ind_{\{v(t,\vec{x})<\infty\}}+\frac{1}{\varepsilon}\ind_{\{v(t,\vec{x})=\infty\}}$. 
	Given $(t,\vec{x})\in\R_+\times\Db$ and an $\Fb^0$-stopping time $\theta$ with values in $(t,\infty)$, an application of Jankov-von Neumann's measurable selection theorem then yields the existence of a universally measurable kernel, $(\Q_\omega)_{\omega\in\Db}$, with $\Q_\omega\in\Pk_{\theta(\omega),\omega}$ and $\E^{\Q_\omega}[G]\ge v^\varepsilon(\theta(\omega),\omega)$, for every $\omega\in\Db$; Galmarino's test implies that $\Q_\cdot$ is $\F^U_\theta$-measurable. 
 	Moreover, according to Lemma \ref{lem:cond_forward_backward}, for any $\P\in\Pk_{t,\vec{x}}$ we have $\P\otimes_\theta\Q_\cdot\in\Pk_{t,\vec{x}}$. In consequence, 
	\begin{equation*}
    v(t,\vec{x})
    \ge \E^{\P\otimes_\theta\Q_\cdot}[G]
    = \E^\P\left[\E^{\Q_\omega}\right]
    \ge \E^\P\left[v^\varepsilon(\theta,X_{\cdot\wedge\theta})\right]. 
  	\end{equation*}		
  	Since $\varepsilon>0$ and $\P\in\Pk_{t,\vec{x}}$ were both arbitrarily chosen, we obtain $v(t,\vec{x})\ge\sup_{\P\in\Pk_{t,\vec{x}}}\E^\P\left[v(\theta,X_{\cdot\wedge\theta})\right]$ which completes the proof.
	\end{proof}
	
	By use of the same arguments as in the proof of Lemma \ref{lem:equiv_weak_mvm}, it follows directly from this result that relation \eqref{eq:dpp_z}, and thus relation \eqref{eq:dpp}, holds for any $\theta$ which is a stopping time in the raw filtration generated by the Brownian motion. Theorem \ref{thm:dpp} then follows from the fact that any $\Fb$-stopping time is predictable, and that for any $\Fb$-predictable time $\theta$, there exists a predictable time $\tilde\theta$ in the raw filtration such that $\theta=\tilde\theta$ a.s.

	\appendix
	
	\section{An alternative proof of the DPP} 
	
	A-priori assuming certain continuity of the objective and value-function, we here provide an alternative proof of the dynamic programming principle. Specifically, considering payoff functions which admit a Markovian structure and satisfy \eqref{eq:ass_special_form}, in the spirit of \cite{bouchard2012,bouchard2011} (see also e.g. \cite{aksamit2016}), we circumvent the need for measurable selection theorems by explicitly constructing a measurable optimiser via a covering argument. Although the result is a special case of Theorem \ref{thm:dpp}, we choose to report this independent argument for we find it of interest.
	
	Throughout the appendix we apply the convention of writing $\xi$ for elements in $\Pc$ and $\pmb\xi$ for MVMs; in integrals we sometimes omit the subscript on MVMs and they are then to be understood in the following sense: 
 	$\int...\;\pmb\xi(\ds)=\int...\;\pmb\xi_\infty(\ds)=\int...\;\di A^{\pmb\xi}_s$. 
 	Further, integrals without limits are to be understood as taken from zero to infinity. 
	
	 We say that a cost function $c$ is of \emph{Markov type}, if it admits the representation
	 \begin{align*}
	 c\left(W_{\cdot\wedge\tau},\tau\right)=f\left(X_\tau\right), \;\;\textrm{a.s.}, 
	\end{align*}	  
 	for some $\Fb$-Markov process $(X_t)_{t\ge 0}\in\R^n$ and continuous function $f:\R^n\to\R$; viewed from the perspective of the original problem formulation \eqref{problem}, this ensures a certain Markov structure with respect to the Brownian filtration $\Fb$ (as opposed to $\Gb$). 	 
	For the class of cost functions of Markov type, we then introduce the function $\tilde v:S\to\R$, where $S:=\{(t,x,\xi)\in\R_+\times\R^n\times\Pc:\supp\xi\subseteq[t,\infty)\}$ by
		\begin{align}\label{eq:value-function_app}
			\tilde v(t,x,\xi)
			:=\sup_{\pmb\xi\in\MVM^t(\xi)}\E\left[\int f\left(X^{t,x}_s\right)\di A^{\pmb\xi}_s\right],		
		\end{align}
	where $\MVM^t(\xi)$ 
	is the set of continuous adapted MVMs with $\pmb\xi_t=\xi$ 
	which are independent of $\F_t$. Notably, this function remains unaltered if relaxing the condition of independence of $\F_t$ (cf. \eqref{eq:pseudo_final} below and e.g. Proposition 2.4 in \cite{claisse2016}). 
	We stress that the integral in \eqref{eq:value-function_app} is taken from zero rather than $t$, which implies that for a general point in $S$, the function $\tilde v$ differs from the value-function introduced in \eqref{eq:value-function} in that a potential atom at time $t$ contributes to its value. However, for any point $(t,x,\xi)\in S_0:=\{(t,x,\xi)\in S:\supp\xi\subseteq(t,\infty)\}$, we have that it coincides with the value function; i.e. then $\tilde v(t,x,\xi)=v(t,x,\xi):=\sup_{\pmb\xi\in\MVM^t(\xi)}\E[\int_t^\infty f\left(X^{t,x}_s\right)\di A^{\pmb\xi}_s]$.

	We let $l_\tau$ denote the left end-point of the support of $\tau$, i.e. $l_\tau:=\sup\{t\ge 0: \P(\tau<t)=0\}$, and for any $\varepsilon>0$ and $\eta\in\Pc$, we write $\Tc_\varepsilon(\eta):=\cup_{\xi\in B_\varepsilon(\eta)}\Tc(\xi)$. 
	We may now prove the DPP under the following additional assumption\footnote{
	For the cost functions given in Example \ref{ex:ex_special_form}, $\varphi$ is independent of $\tau$, and thus the last part of Assumption \ref{ass:app} holds trivially. 
	Moreover, for the case when $X_\cdot$ is homogeneous in time, or $\theta$ takes its values in a countable subset, the second last assumption can be weakened to $x\mapsto\E[f(X^{t,x}_{\tau})]$ lower-semicontinuous for any fixed $t>0$, $\tau\in\Tc^t$. In this case the proof simplifies even further since it is then sufficient to cover $\R^n\times\Pc$ as opposed to $S$.}:
		
	\begin{assumption}\label{ass:app}
		Suppose that the cost function $c$ is of Markov type with $f$ locally bounded; 
		that $\tilde v$ is continuous on $S$; 
		and that for any $\eta\in\Pc$, there exists $\varepsilon>0$ and a modulus of continuity $\varphi$ such that:
		$(t,x)\mapsto\E[f(X^{t\wedge l_\tau,x}_{\tau})]$ is continuous on $\R_+\times\R^n$ uniformly w.r.t. $\tau\in\Tc_\varepsilon(\eta)$,
		and $c$ satisfies Assumption \ref{ass:special_form} w.r.t. $\varphi$ for all $\tau\in\Tc_\varepsilon(\eta)$. 
	\end{assumption}

	\begin{proof}[Proof of Theorem \ref{thm:dpp} under Assumption \ref{ass:app}:]	
		Given $\delta>0$, for $(s,z,\eta)\in S_0$, let $s_0>s$ such that $\eta((0,s_0])<\delta$, and let $\bar t\in(s,\frac{s+s_0}{2})$. In turn, define the open set
			\begin{equation*}
				\Ac^{(s,z,\eta)}:=(s-\delta,\bar t)\times B_\delta(z)
				\times\left\{\xi\in\Pc:\Wc_1(\eta,\xi)<\delta(\bar t-s)\right\}.
			\end{equation*}		
		For any $(t,x,\xi)\in\Ac^{(s,z,\eta)}$, it then holds (writing $\delta$ for any multiple of $\delta$) that $\xi((0,\bar t])<\delta$ and thus $\Wc_1(\xi^{|_{\bar t}},\xi)<\delta$. 
		In consequence, there exists $\bar\xi\in \Pc$ with $\supp\bar\xi\subseteq [\bar t,\infty)$, such that $\xi^{|_{\bar t}}\in\Rc(\bar\xi)$ and $\Wc_1(\bar\xi,\xi^{|_{\bar t}})<\delta$, for any $(t,x,\xi)\in\Ac^{(s,z,\eta)}$; in particular, $\Wc_1(\bar\xi,\eta)<\delta$ and $(\bar t,x,\bar\xi)\in S$.
		It follows that for any $\varepsilon>0$, by choosing $\delta\vee\bar t$ small enough, we may ensure that for all $(t,x,\xi)\in\Ac^{(s,z,\eta)}$, 
		\begin{enumerate}
			\item[(\emph{i})]\label{item:1} 
			$|t-\bar t|\vee\bar\varphi\big(\Wc\big(\xi^{|_{\bar t}},\bar \xi\big)\big)<\varepsilon$ 
			and 
			$\xi((0,\bar t])\big(|f(x)|\vee |\tilde v(t,x,\xi)|\vee 1\big)<\varepsilon$,
		\end{enumerate}
		where $\bar\varphi$ is a modulus of continuity such that \eqref{eq:ass_special_form} holds for $\rho\ge \tau\in\Tc(\xi)$, for all $\xi\in\Ac^{(s,z,\eta)}|_{\Pc}$. 
		Moreover, using the continuity properties ensured by Assumption \ref{ass:app}, and choosing if necessary $\delta\vee\bar t$ even smaller, we may ensure\footnote{The reason we need to a-priori assume continuity as opposed to only semi-continuity, is that we need properties (i)-(iii) for the reference point $(\bar t,z,\bar \xi)$ (as opposed to $(s,z,\eta)$). One way to circumvent this would be to consider the topology on $\Pc$ generated by open sets of the form $B_\delta(\eta):=\left\{\xi\in\Pc: \xi\in\mathcal{R}(\eta),\;\Wc_1(\eta,\xi)<\delta\right\}$; it is however not clear to the authors whether $S$ is $\sigma$-compact or even Lindel\"of under this topology.} that also the following properties hold on $\Ac^{(s,z,\eta)}$: 
		\begin{enumerate}
			\item[(\emph{ii})]\label{item:2} $\left|\E\left[\int f(X^{t,x}_u)-f(X^{\bar t,z}_u)\di A^{\pmb\xi}_u\right]\right|\le\varepsilon$, for any $\pmb\xi\in\MVM^{\bar t}(\bar\xi)$; 
			\item[(\emph{iii})]\label{item:3} $|\tilde v(t,x,\xi)-\tilde v(\bar t,z,\bar \xi)|\le\varepsilon$. 
		\end{enumerate}

	The collection of sets $\{\Ac^{(s,z,\eta)}:(s,z,\eta)\in S_0\}$ provides an open cover of $S_0$. Further, since $S_0$ is a subspace of a Polish space and therefore Lindel\"of, this cover admits a countable subcover; we denote the latter by $(B_i)_{i\in\Nb}$, and for each $i\in\Nb$, we define the associated reference point via the identification $(t_i,x_i,\xi_i)\;\widehat =\; (\bar t,z,\bar\xi)$ and let $\varphi_i \;\widehat =\; \bar\varphi$; notably the reference points lie in $S$ but not necessarily in $S_0$. 
	We then obtain a measurable partition of $S_0$ as follows: 
			\begin{equation*}
				A_1=B_1\cap S_0,\quad A_{j+1}=\left(B_{j+1}\setminus(B_1\cup\dots \cup B_j)\right)\cap S_0,\quad j\ge 1. 
			\end{equation*}
	Further, we define an associated family of MVMs $(\pmb\xi^i)_{i\in\Nb}$ such that $\pmb\xi^i\in\MVM^{t_i}(\xi_i)$ and 
		\begin{align*}
			\E\Big[\int f\left(X^{t_i,x_i}_s\right)\di A^{\pmb\xi^i}_s\Big]
			\ge \tilde v(t_i,x_i,\xi_i)-\varepsilon. 
		\end{align*}
		
	For a general $(t,x,\xi)\in A_i$, let $\Gamma^{i,\xi}$ such that $\Wc(\xi_i,\xi^{|_{t_i}})=\iint |s-u|\Gamma^{i,\xi}(\di s,\di u)$; we denote by $m^{i,\xi}(s,\cdot)$ the family of disintegrated kernels for which $\Gamma^{i,\xi}(\di s,\di u)=\xi_i(\di s)m^{i,\xi}(s,\di u)$. We then define $\pmb\xi^{i,\xi}$ as follows: $\pmb\xi^{i,\xi}_{(t\vee\cdot)\wedge t_i}(A):=\xi(A)$, and 
		\begin{align*}
			\pmb\xi^{i,\xi}_{\cdot\vee t_i}(A)
			:=\xi\big(A\cap(t,t_i]\big)
			+\xi\big((t_i,\infty)\big)\int \pmb\xi^i_\cdot(\di u)m^{i,\xi}\big(u,A\big), 
			\;\; A\in\Bc(\R_+).
		\end{align*}
	We note that this indeed defines a MVM in $\MVM^t(\xi)$.
	By use of Assumption \ref{ass:special_form} we then obtain
			\begin{align*}
				\Big|\E\Big[\int_t^{t_i} f\left(X^{t,x}_s\right)-f\left(x\right)\xi(\di s)\Big]\Big|
				\le \varphi_i\Big(\int_t^{t_i} \left|s-t\right|\xi(\di s)\Big)
				\le \varphi_i\big(\xi((t,t_i])|t-t_i|\big),
			\end{align*}
	which using that $(t,x,\xi)\in A_i$ and property (i) implies $\E[\int_t^{t_i} f(X^{t,x}_s)\pmb\xi^{i,\xi}(\di s)]\ge \xi((t,t_i])f(x)-\varepsilon>-\varepsilon$.
	Moreover, using again Assumption \ref{ass:special_form} and (i), we also obtain (cf. \eqref{eq:proof_cont}), 
			\begin{align*}
				\bigg|\E\Big[\iint  &f\left(X^{t,x}_u\right)\pmb\xi^i(\di s)m^{i,\xi}\big(s,\di u\big)
				-\int f\left(X^{t,x}_s\right)\pmb\xi^i(\di s)\Big]\bigg|\\
				&\le 
				\varphi_i\left(\E\Big[\Big|\iint u-s\; \pmb\xi^i(\di s)m^{i,\xi}\big(s,\di u\big)
				\Big|\Big]\right)\\
				&\le\varphi_i\left(\iint |u-s|m^{i,\xi}(s,\di u)\xi_i(\di s)\right)
				\;\le\; \varphi_i\left(\Wc\left(\xi_i,\xi^{|_{t_i}}\right)\right)<\varepsilon.
			\end{align*} 
	On the other hand, using property (ii) we obtain
			\begin{align*}	
				\E\Big[\int f\left(X^{t,x}_s\right)\pmb\xi^i(\di s)\Big]
				\ge 
				\E\Big[\int f\left(X^{t_i,x_i}_s\right)\pmb\xi^i(\di s)\Big]-\varepsilon
				\ge 
				\tilde v(t_i,x_i,\xi_i)-\varepsilon.
			\end{align*}	
	Further, recall that by (iii), $\tilde v(t_i,x_i,\xi_i)\ge v(t,x,\xi)-\varepsilon$.
	Combining the above and using once again (i), we thus obtain that for $(t,x,\xi)\in A_i$, 
			\begin{align}\label{eq:proof_dpp_no_meas_bound}	
				\E\Big[\int f\left(X^{t,x}_s\right)\pmb\xi^{i,\xi}(\di s)\Big]
				\ge 
				\xi\big((t_i,\infty)\big)
				(v(t,x,\xi)-\varepsilon)-\varepsilon
				\ge
				v(t,x,\xi)-\varepsilon.
			\end{align}

		We may now conclude in a standard way. First, we may, w.l.o.g., take $(t,x,\xi)=(0,x,\mu)$. 
		Given $\pmb\xi\in\MVM(\mu)$, we then introduce the modified MVM $\pmb\xi^\varepsilon$ as follows: 
		$\pmb\xi^\varepsilon_{\cdot\wedge \theta}:=\pmb\xi_{\cdot\wedge\theta}$, and
			\begin{align*}
				\pmb\xi^\varepsilon_{\cdot\vee\theta}(A) :=
				\pmb\xi_\theta\big(A\cap(0,\theta]\big)
				+\pmb\xi_\theta\big((\theta,\infty)\big)
				\sum_{i\in\Nb}
				\ind_{\big\{\big(\theta,X^x_\theta,\pmb\xi^{|_\theta}_\theta\big)\in A_{i}\big\}}
				\pmb\xi^{i,\pmb\xi^{|_\theta}_\theta}_{\cdot\vee\theta}(A),
				\;A\in\Bc(\R_+).
			\end{align*}
		We note that the thus defined process is a well-defined continuous and adapted MVM with $\pmb\xi^\varepsilon_0=\mu$.
		Moreover, by use of \eqref{eq:proof_dpp_no_meas_bound}, the fact that $X$ is a Markov process and that $\pmb\xi^i$ is independent of $\F_{t_i}$ (this is why we need $t\le t_i$ for $(t,x,\xi)\in A_i$), we obtain the following on $\{(\theta,X^x_\theta,\pmb\xi^{|_\theta}_\theta)\in A_i\}$:  
			\begin{align*}
				\E\Big[\int f\left(X^x_s\right)\pmb\xi^{i,\pmb\xi^{|_\theta}_\theta}(\di s)|\F_\theta\Big]
				=\E\Big[\int f\left(X^{t,x}_s\right)\pmb\xi^{i,\xi}(\di s)\Big]_{\substack{(t,x,\xi)\;=\\ \big(\theta,X^x_\theta,\pmb\xi^{|_\theta}_\theta\big)}}
				\ge v\left(\theta,X^x_\theta,\pmb\xi^{|_\theta}_\theta\right)-\varepsilon. 
			\end{align*}			
		Hence, we obtain 
			\begin{align*}
				&v(0,x,\mu)
				\ge 
				\E\left[\int_0^\infty f(X^x_s)\di A^{\pmb\xi^\varepsilon}_s\right]\\
				& = \E\bigg[\int_0^\theta f(X^x_s)\pmb\xi_\theta(\di s)
				+\pmb\xi_\theta\big((\theta,\infty)\big)
				\sum_{i\in\Nb}\ind_{\left\{\left(\theta,X^x_\theta,\pmb\xi^{|_\theta}_\theta\right)\in A_{i}\right\}}
				\E\Big[\int f\left(X^x_s\right)\pmb\xi^{i,\pmb\xi^{|_\theta}_\theta}(\di s)|\F_\theta\Big]\bigg]\\	
				& \ge 
				\E\bigg[\int_0^\theta f(X^x_s)\pmb\xi_\theta(\di s)
				+\pmb\xi_\theta\big((\theta,\infty)\big)
				v\left(\theta,X^x_\theta,\pmb\xi^{|_\theta}_\theta\right)\bigg]-\varepsilon.
			\end{align*}		
		Since $\varepsilon>0$ and $\pmb\xi\in\MVM(\mu)$ were both arbitrarily chosen, this yields the first inequality.

		In order to show the reverse inequality, let $\pmb\xi\in\MVM(\mu)$ and let $\theta$ be a finite $\Fb$-stopping time. We note that in our case the so-called pseudo Markov property (cf. \cite{claisse2016}) holds trivially, and for almost all $\omega\in\Omega$ we have that
			\begin{align}\label{eq:pseudo_final}
				\E\Big[\int_\theta^\infty f(X^x_s)\pmb\xi^{|_\theta}(\ds)|\Fc_\theta\Big](\omega)
				&=		\int_\Omega\int_{\theta(\omega)}^\infty f\left(X^{\theta(\omega),X^x_{\theta(\omega)}}_s(\tilde\omega)\right)\pmb{\tilde\xi}^{\theta(\omega),\omega}(\tilde\omega;\ds)\di \mathbb{W}(\tilde\omega)\nonumber\\
				&\le 
				v\left(\theta(\omega),X^x_{\theta(\omega)}(\omega),\pmb\xi^{|_{\theta(\omega)}}_{\theta(\omega)}(\omega)\right),
			\end{align}			 
		where 
				$\pmb{\tilde\xi}^{\theta(\omega)\omega}_u(\tilde\omega)
				=
				\pmb\xi^{|_{\theta(\omega)}}_{\theta(\omega)}(\omega)\ind_{\{u<\theta(\omega)\}}
				+
				\pmb\xi^{|_{\theta(\omega)}}_u\left(\omega*_{\theta(\omega)}\tilde\omega\right)\ind_{\{u\ge\theta(\omega)\}}$				
		with the concatenated path given by $\omega*_t\tilde\omega(s)=\ind_{\{0\le s< t\}}\omega(s)+\ind_{\{t\le s\}}(\omega(t)+\tilde\omega(s)-\tilde\omega(t))$; indeed, for $\omega\in\Omega$ fixed, $\pmb{\tilde\xi}^{\theta(\omega),\omega}(\cdot)$ is independent of $\Fc_\theta$ and thus lies in $\MVM^{\theta(\omega)}\big(\pmb\xi^{|_{\theta(\omega)}}_{\theta(\omega)}(\omega)\big)$. 
		Hence, we obtain
			\begin{align*}
				\E\Big[\int_0^\infty f(X^x_s)\di A^{\pmb\xi}_s\Big]
				&=
				\E\Big[\int_0^\theta f(X^x_s)\pmb\xi_\theta(\ds)+\pmb\xi_\theta\big((\theta,\infty)\big)\E\Big[\int_\theta^\infty f(X^x_s)\pmb\xi^{|_\theta}(\ds)|\Fc_\theta\Big]\Big]\\
				&\le 
				\E\Big[\int_0^\theta f(X^x_s)\pmb\xi_\theta(\ds)+\pmb\xi_\theta\big((\theta,\infty)\big)v\left(\theta,X^x_\theta,\pmb\xi^{|_\theta}_\theta\right)\Big],
			\end{align*}
		which completes the proof since $\pmb\xi$ and $\theta$ were arbitrarily chosen. 
	\end{proof}

\bibliography{refs}
\bibliographystyle{plaintest}

\end{document}